\documentclass[american]{amsart}
\usepackage{amsmath, amsfonts, amssymb, amscd, amsthm, amsbsy, esint, bbm, epsf, mathtools}
\usepackage[utf8]{inputenc}

\usepackage{datetime}
\usepackage[margin=1in]{geometry}
\usepackage{comment}

\mathtoolsset{showonlyrefs}

\usepackage[colorlinks,linkcolor = blue,citecolor = green]{hyperref}

\newtheorem{theorem}{Theorem}[section]

\newtheorem{corollary}{Corollary}
\newtheorem{lemma}{Lemma}[section]
\theoremstyle{definition}
\newtheorem{definition}{Definition}
\newtheorem{remark}{Remark}

\numberwithin{equation}{section}

\DeclareMathOperator{\R}{\mathbb{R}}

\DeclareMathOperator{\bd}{\partial}

\DeclareMathOperator{\embeds}{\hookrightarrow}

\DeclareMathOperator*{\essinf}{ess\,inf}

\DeclareMathOperator{\eps}{\varepsilon}

\newcommand{\norm}[1]{\left|\left| #1 \right|\right|}
\newcommand{\abs}[1]{\left| #1 \right|}
\newcommand{\set}[1]{\left\{ #1 \right\}}

\title{Nonlinear Asymptotic Stability in $L^\infty$ for Lipschitz Solutions to Scalar Conservation Laws}
\author{William Golding}
\address[William Golding]{\newline Department of Mathematics, \newline The University of Texas at Austin, Austin, TX 78712, USA}
\email{wgolding@utexas.edu}

\date{\today}

\thanks{\textbf{Acknowledgments:} The author would like to thank his advisor, Alexis Vasseur, for several helpful conversations and for continued encouragement and support.}
\thanks{\textbf{Funding:} W. Golding is partially supported by NSF grant DMS 1840314.}

\begin{document}

\begin{abstract}
In this note, we show nonlinear stability in $L^\infty$ for Lipschitz solutions to genuinely nonlinear, multi-dimensional scalar conservation laws. As an application, we are able to compute algebraic decay rates of the $L^\infty$ norm of perturbations of global-in-time Lipschitz solutions, including perturbations of planar rarefaction waves. 
Our analysis uses the De Giorgi method applied to the kinetic formulation and is an extension of the method introduced recently by Silvestre in [\textit{Comm. Pure Appl. Math}, 72(6):1321–1348, 2019].
\end{abstract}

\maketitle

\tableofcontents

\section{Introduction}
In this note, we study multi-dimensional scalar conservation laws of the form: 
\begin{equation}\label{eqn:scalar}
\partial_t u + \nabla_x \cdot A(u) = 0,
\end{equation}
where $(t,x)\in \R^+ \times \R^n$ are time and space, respectively, and $u: \R^+ \times \R^n \rightarrow \R$ is the unknown conserved quantity with flux $A: \R \rightarrow \R^n$. We only consider solutions $u$ which satisfy \eqref{eqn:scalar} in the sense of distributions; are further bounded in that $u(t,x) \in I = [-\|u\|_{L^\infty},\|u\|_{L^\infty}]$; and are entropic in the sense that
\begin{equation}\label{eqn:entropy}
\partial_t\abs{u - k} + \nabla_x \cdot \left[ \mathrm{sgn}(u-k)(A(u)- A(k)) \right]\le 0,
\end{equation}
in the sense of distributions, for each $k \in I$. Throughout the following, we will assume a standard genuine nonlinearity condition on the flux, namely:
\begin{definition}\label{defn:genuine_nonlinearity}
We say a flux $A$ is {\bf genuinely nonlinear} (on an interval $I = [a,b]$) with parameters $C_0, \ C_1 > 0$ and $\alpha \in (0,1]$ if $a = \nabla \cdot A \in C^1(I)$ with $\|a\|_{L^\infty(I)} \le C_1$ and
\begin{equation}\label{eqn:nondegeneracy}
\abs{\set{v \in I \mid |t + a(v)\cdot \xi| < \delta}} \le C_0 \delta^\alpha, \qquad \text{ for each }t^2 + |\xi|^2 = 1.
\end{equation}
\end{definition}
Some of the simplest examples of genuinely nonlinear fluxes are when $A:\R \rightarrow \R$ is a scalar, strictly convex function, in which case $\alpha = 1$. Definition \ref{defn:genuine_nonlinearity} is one standard generalization of such fluxes to the multi-dimensional setting that also allows the flux to be degenerate at some points (i.e. $A:\R \rightarrow \R$ can have inflection points when $\alpha < 1$). Throughout, we have in mind the prototypical genuinely nonlinear flux, the multi-dimensional Burgers flux:
\begin{equation}\label{eqn:multid_burgers}
A(u) = (A_1(u), \hdots, A_n(u)) = \left(u^2/2, u^3/3,  \hdots, u^{n+1}/(n+1)\right),
\end{equation}
which is genuinely nonlinear with exponent $\alpha = 1/n$.

We mention that in this general setting, existence of global-in-time solutions in $L^1 \cap L^\infty$ is classical and was shown by Kru\v{z}kov using using the vanishing viscosity method. We mention that in one space dimension, one can additionally use the method of compensated compactness as done by Tartar in \cite{tartar}. Uniqueness and nonlinear stability in $L^1$ are also classical and are shown using the doubling of variables method introduced by Kru\v{z}kov in \cite{kruzkov}. However, Kru\v{z}kov's result relies upon the large family of entropies $\abs{u - k}$ defined in \eqref{eqn:entropy} and is known to breakdown for systems. Subsequently, in \cite{dafermos,diperna} Dafermos and DiPerna introduced the relative entropy method, which provides nonlinear stability of Lipschitz solutions in the $L^2$ norm, with the advantage that it relies upon a single convex entropy and is quite amenable to the study of systems. Our main result is similar to that of Dafermos and DiPerna, but in the $L^\infty$ norm and restricted to the scalar setting.

More precisely, our main theorem is a generalization of a recent result of Silvestre in \cite{silvestre}:
\begin{theorem}\label{thm:stability}
Suppose $A:\R^n \rightarrow \R$ is a genuinely nonlinear flux in the sense of Definition \ref{defn:genuine_nonlinearity} with parameters $C_0$, $C_1$, and $\alpha$ on the interval $[-\Lambda, \Lambda]$ for some $\Lambda \in \R^+$.
Suppose $u$ and $\tilde u$ are bounded entropy solutions to \eqref{eqn:scalar} on $[s,T] \times \Omega$ for $\Omega$ a bounded, open subset of $\R^n$ and $s < T$ with $\|u\|_{L^\infty}, \|\tilde u \|_{L^\infty} \le \Lambda$ and $\|\nabla \tilde u\|_{L^\infty} \le \Gamma$. Then, for any $\widetilde \Omega$ compactly contained in $\Omega$, and any $s < t < T$,
\begin{equation}\label{eq:main_estimate}
\|u - \tilde u\|_{L^\infty([t,T] \times \widetilde \Omega)} \le \frac{C(\Lambda, C_0, C_1,\alpha, n)(1 + d\,\Gamma)^{\eta}}{d^{(n+1)\gamma}}\|u - \tilde u\|_{L^1([s,T]\times \Omega)}^{\gamma},
\end{equation}
for $\gamma = \gamma(\alpha,n) \in (0,1)$, $\eta = \eta(\alpha, n) \in (0,\alpha/2)$, and $d = \min(t-s,dist(\bd\Omega, \widetilde \Omega))$.
\end{theorem}

\begin{remark}
    Both exponents $\gamma$ and $\eta$ are explicitly computable, and are computed in the main step of the proof in Section \ref{sec:main}. However, the exponents $\gamma$ and $\eta$ in Theorem \ref{thm:stability} are usually far from optimal. At present, we care only that there exist some positive exponents satisfying \eqref{eq:main_estimate}, which will allow us to deduce some asymptotic stability. 
\end{remark}

Theorem \ref{thm:stability} was shown by Silvestre in \cite{silvestre}, when $\tilde u$ is a constant. More precisely, Silvestre's result yields the estimate \eqref{eq:main_estimate} for solutions $u$ which satisfy $u(t,x) - m \in L^1\cap L^\infty$, for some $m\in \R$. This provides $L^\infty$ stability with explicit algebraic decay rates of $u$ to $m$ as $t\rightarrow \infty$, with an optimal rate if $m = 0$. In \cite{serre_silvestre}, Serre and Silvestre removed the restriction that $u(t,x) \in L^\infty$, at least for the generalized Burgers equation, using the compensated integrability theory developed by Serre in \cite{serre1, serre2}.

As a consequence of Theorem \ref{thm:stability}, the $L^1$ contraction property, and the maximum principle, we deduce the following decay rates for general scalar conservation laws:
\begin{corollary}\label{cor:rate}
Suppose $A:\R^n \rightarrow \R$ is a genuinely nonlinear flux in the sense of Definition \ref{defn:genuine_nonlinearity} with parameters $C_0$, $C_1$, and $\alpha$ on the interval $[-\Lambda, \Lambda]$ for some $\Lambda \in \R^+$.
Suppose $u$ and $\tilde u$ are bounded entropy solutions to \eqref{eqn:scalar} on $\R^+ \times \R^n$ with initial data $u_0$ and $\tilde u_0$, respectively.
Suppose that $u_0 - \tilde u_0 \in L^1(\R^n)$, $\tilde u$ is Lipschitz, and further
\begin{equation}
    \norm{u_0}_{L^\infty(\R^n)} \le \Lambda, \quad\norm{\tilde u_0}_{L^\infty(\R^n)} \le \Lambda, \quad\text{ and } \quad \sup_{t \in \R^+} t\|\nabla \tilde u(t)\|_{L^\infty(\R^n)} \le \overline{\Gamma}.
\end{equation}
Then,
\begin{equation}\label{eqn:rate}
\|u(t) - \tilde u(t)\|_{L^\infty(\R^n)} \le \frac{C(\Lambda, C_0, C_1, \alpha, n) (1 + \overline{\Gamma})^\eta}{t^{n\gamma}}\|u_0 - \tilde{u}_0\|_{L^1(\R^n)}^\gamma,
\end{equation}
where $\gamma$ and $\eta$ are the exponents from Theorem \ref{thm:stability}.
\end{corollary}

Prior to Silvestre's result \cite[Theorem 1.5]{silvestre}, decay rates in $L^\infty$ were largely limited to particular one dimensional equations \cite{dafermos3, lax}. Additionally, $L^p$ decay rates for $1 \le p < \infty$ were limited to periodic solutions in multiple space dimensions \cite{chen_frid2,dafermos2,debussche_vovelle,enquist_e,panov2}. While Silvestre's result provided a major extension of these preceding results, Silvestre's result does not apply to solutions that tend asyptotically towards a non-constant profile.

Importantly, in Corollary \ref{cor:rate}, neither $u$ nor $\tilde u$ needs to belong to $L^1(\R^n)$, only their difference. As a prototypical example, Corollary \ref{cor:rate} applies when $\tilde u$ is a planar rarefaction wave, which is \emph{not} covered by the results of \cite{silvestre}. The stability of planar rarefaction waves has been studied extensively in the $L^2$ setting by means of the relative entropy method (for instance, see \cite{serre3}) and by extension in $L^p$ for $1 \le p < \infty$. We refer the reader to \cite[Chapter 5]{dafermos_book} for a thorough overview of such results. Additionally, stability of planar rarefaction waves has been studied in $L^\infty$ in various viscous settings, but not uniformly in the viscosity parameter. Finally, in the one dimensional setting, more is known: For fluxes with at most a single inflection point, the optimal decay rate is known using a study of generalized characteristics \cite{dafermos3} and the references therein (see also \cite[Chapter 11]{dafermos_book}). However, this method relies heavily upon the ordering of the real line and has not been extended to multidimensional settings. To our knowledge, Corollary \ref{cor:rate} appears to be the first asymptotic stability result for planar rarefaction waves in the $L^\infty$ norm, at least in the purely hyperbolic setting in multiple space dimensions. 

\begin{remark}
    The assumption that $\norm{\nabla \tilde u(t)}_{L^\infty(\R^n)} = O(t^{-1})$ is rather reasonable by scaling considerations. For the prototypical global-in-time Lipschitz solutions, namely planar rarefaction waves, this is the explicitly computable rate of decay. Additionally, this rate of decay holds in the special case $n = 1$ and $A:\R \rightarrow \R$ a strictly convex flux. Indeed, $\tilde u(t,x)$ is increasing in $x$ (because it is globally Lipschitz) and the well-known Oleinik estimate, originally shown in \cite{oleinik}, implies this rate of decay, with an explicitly computable constant $\overline{\Gamma} = \inf_{u\in [-\Lambda,\Lambda]} A^{\prime\prime}(u)$. In addition, one can use the ``asymptotic Oleinik estimate'' shown in \cite{jenssen_sinestrari} for fluxes with a single inflection point to obtain the existence of $\overline{\Gamma}$ for the fluxes $A(u) = u^p$, when $p \ge 2$. However, in general, it is unclear whether the existence of $\overline{\Gamma}$ is guaranteed. Oleinik type estimates are known to fail even in the one dimensional case when the flux has just two inflection points, and additionally for most multi-dimensional fluxes \cite{hoff}. Furthermore, in \cite{crippa_otto_westdickenberg, delellis_otto_westdickenberg}, the authors show that $C^{0,\alpha}$ estimates with $\alpha = 1/n$, the genuine nonlinearity exponent, hold and are optimal for general entropy conservative solutions to the multi-dimensional Burgers equation. This suggests the difficulty of obtaining quantitative estimates on the gradient.
\end{remark}


\begin{remark}
    Perhaps surprisingly, we are not able to sharpen the exponent $\gamma$ using scaling, even for homogeneous conservation laws like the multidimensional Burgers equation, where the flux $A(u)$ is defined in \eqref{eqn:multid_burgers}. For the flux $A$, there is a two-parameter scaling which allows us to track the dependence of \eqref{eqn:rate} on $\Lambda$ and obtain
    \begin{equation}\label{eqn:full_estimate}
        \|u(t) - \tilde{u}(t)\|_{L^\infty(\R^n)} \le \frac{C(n)(1 + \Lambda^{n-1}\overline{\Gamma})^{\eta}\Lambda^{1 - \gamma/\gamma_0}\|u_0 - \tilde{u}_0\|_{L^1(\R^n)}^\gamma}{t^\gamma},
    \end{equation}
    where $\gamma_0 = \left(1 + \frac{n(n+1)}{2}\right)^{-1}$ is the same as in \cite{silvestre}, and should be the optimal decay rate. However, the quantity $\Lambda$ appears on the right hand side, which we expect to stay order $1$ for all time, preventing us from iteratively improving the estimate. Unfortunately, it seems we cannot use our method to replace $\Lambda$ with a quantity like the $L^\infty$ oscillation which is decaying (at least locally).
\end{remark}

The remainder of the note is organized as follows: In Section \ref{sec:prelim}, we introduce some preliminary results needed in the proof of Theorem \ref{thm:stability}, namely the kinetic formulation of \eqref{eqn:scalar} and the associated averaging lemmas. We also briefly summarize the main new idea behind Theorem \ref{thm:stability}. In Section \ref{sec:main}, we provide a complete proof of Theorem \ref{thm:stability}. In Section \ref{sec:applications}, we provide a short proof of Corollary \ref{cor:rate}.



\section{Preliminaries and Proof Idea}\label{sec:prelim}

\subsection{Preliminaries}
For bounded entropy solutions to \eqref{eqn:scalar}, we recall that the associated kinetic function $f$ is of the form
\begin{equation}\label{eqn:structure}
f(t,x,v) = \begin{cases} \chi_{[0,u(t,x)]}(v) \quad \text{if }u(t,x) \ge 0\\ -\chi_{[u(t,x),0]}(v) \quad \text{if }u(t,x) \le 0\end{cases}.
\end{equation}
By the work of Lions, Perthame, and Tadmor (see \cite{lions_perthame_tadmor1}), this kinetic function satisfies the equation
\begin{equation}\label{eqn:kinetic}
\partial_t f + a(v)\cdot\nabla_x f = \partial_v\mu,
\end{equation}
where $\mu$ is a (unique) locally finite Borel measure on $\R^{n+2}$. Moreover, it is clear that $f$, and consequently $\mu$, are supported in $\R^{n+1} \times [-\|u\|_{L^\infty},\|u\|_{L^\infty}]$ and $u$ is recovered from $f$ via the relation
\begin{equation}\label{eqn:kinetic_relation}
u(t,x) = \int_{\R} f(t,x,v) \;dv.
\end{equation}

We will make heavy use of the following relatively standard averaging lemma. The version used can be found with a proof in \cite{tadmor_tao}. 
\begin{theorem}[\protect{\cite[Averaging Lemma 2.1]{tadmor_tao}}]\label{thm:averaging}
Suppose $f \in L^2(\R^{n+1}\times \R_v)$ satisfies the equation
$$\partial_t f + a(v)\cdot \nabla_x f = \partial_v\mu_1 + \mu_0 + g,$$
 where $g\in L^1(\R^{n+1}\times \R)$, and both $\mu_1$ and $\mu_0$ are Borel measure with total finite variation. Assume further that $a(v)$ satisfies the non-degeneracy condition \eqref{eqn:nondegeneracy} with parameters $\alpha$, $C_0$. Then, for any $\psi(v) \in L^\infty(\R_v)$, $s = \frac{\alpha}{\alpha + 4}$,
\begin{equation}
\left\|\int \psi(v)f(t,x,v)\;dv\right\|_{W^{\theta, r}(\R^{n+1})} \lesssim_{C_0,\alpha, \psi} \left\|f\right\|_{L^2(\R^{n+2})}^{1-\theta}\left(\|g\|_{L^1} + \|\mu_0\|_{TV} + \|\mu_1\|_{TV}\right)^{\theta},
\end{equation}
where $\theta \in (0,s)$ and $r = \frac{2}{1 + \theta}$.
\end{theorem}

\subsection{Proof Idea}

The main idea underlying the proof Theorem \ref{thm:stability} is to bypass \eqref{eqn:scalar} and to work directly on the kinetic formulation given by \eqref{eqn:kinetic}. This is motivated by the fact that if $u$ and $\tilde u$ are two non-negative solutions to \eqref{eqn:scalar}, $w:= u - \tilde u$ solves an equation of the form \eqref{eqn:scalar} with a new flux $\tilde A$, 
which is only $(t,x)$-independent if $\tilde u$ is identically constant. However, if $f$ and $g$ are the kinetic functions of $u$ and $\tilde u$, respectively, with corresponding measures $\mu_1$ and $\mu_2$, linearity of \eqref{eqn:kinetic} immediately implies $h:=f - g$ solves
\begin{align}
&\partial_t h + a(v) \partial_x h = \partial_v (\mu_1 - \mu_2)\\
&h(t,x,v)= \begin{cases} 1 &\text{if }u(t,x) > v > \tilde u(t,x) > 0\\
	-1 &\text{if } \tilde u(t,x) > v > u(t,x) > 0\\
	0 &\text{otherwise} \end{cases}.
\end{align} 
We see that $h$ solves an equation very similar to \eqref{eqn:kinetic}, but at the expense of some rigidity on the measure $\mu_1 -\mu_2$ (which now lacks a sign) and on the ``shape of the equilibrium functions''. To retain enough structure to apply the De Giorgi method of \cite{silvestre}, we will need to assume $\mu_2 = 0$ and, additionally, that $\tilde u$ is Lipschitz. The De Giorgi method was initially introduced in \cite{degiorgi} to study the regularity of solutions to nonlinear elliptic equations. Since then, it has been applied in various fields of partial differential equations far beyond its initial scope. In one such adaptation, Silvestre utilized dispersive averaging lemmas instead of the original elliptic Caccioppoli inequality to examine scalar conservation laws in \cite{silvestre}. Subsequently, the author modified Silvestre's proof to rely solely on the kinetic formulation, which enabled an application of the method to a system of isentropic gas dynamics in \cite{golding}. As we necessarily rely heavily upon the kinetic formulation, we follow the approach from \cite{golding} in the original context of scalar conservation laws. 



\section{Proof of Theorem \ref{thm:stability}}\label{sec:main}

In this section, we prove Theorem \ref{thm:stability}. Throughout, we will treat time identically to the spatial variables and so to simplify notation, we will introduce $\overline{x} = (t,x)$, $\overline{A}(u) = (u, A)$, and $\overline{a} = (1, a)$, which yields the equivalent time-independent formulations of \eqref{eqn:scalar} and \eqref{eqn:kinetic} as
\begin{align}
	\nabla \cdot \overline{A}(u) &= 0 \label{eqn:scalar2}\\
	\overline{a}(v) \cdot \nabla f & = \partial_{v}\mu \label{eqn:kinetic2}.
\end{align}
We work first at the unit scale and derive a local energy estimate for the difference $u - \tilde u$ of two solutions of \eqref{eqn:scalar2}, which will show a gain of integrability. 
To this end, we begin by showing local-in-$v$ control over the total variation of the entropy dissipation measure $\mu$.
\begin{lemma}\label{lem:variation}
Suppose $u$ and $\tilde{u}$ are bounded entropy solutions to \eqref{eqn:scalar2} on $B_2\subset \R^{n+1}$ with $0 \le u, \tilde u \le \Lambda$ on $B_2$ and $\tilde u$ is entropy conservative. Let $f$ and $g$ be the kinetic functions associated to $u$ and $\tilde u$, respectively, and $h = f - g$. Then, for any Lipschitz $w: \R^{n+1} \rightarrow \R$ with $0\le w \le \Lambda$ and $0 < r < R \le 2$, $\mu$ satisfies
\begin{align}
\mu\left( \set{(x,v) \mid x \in B_r, \ v\in (w(x), \infty)}\right) &\le C(\Lambda, n)\|\overline{a}\|_{L^\infty(0,\Lambda)}\left(\|w\|_{Lip(B_R)} + \frac{1}{R-r}\right)\\
    &\qquad\times\int_{B_R}\int_{w(x)}^\infty |h(x,v)| \;dvdx \label{eqn:bound1}\\
\mu\left( \set{(x,v) \mid x \in B_r, \ v\in (-\infty, w(x))}\right) &\le C(\Lambda, n)\|\overline{a}\|_{L^\infty(0,\Lambda)}\left(\|w\|_{Lip(B_R)} + \frac{1}{R-r}\right)\\
    &\qquad\times\int_{B_R}\int_{0}^{w(x)} |h(x,v)| \;dvdx.\label{eqn:bound2}
\end{align}
\end{lemma}

\begin{proof}
Let us first note that $\tilde u$ conserves entropy means exactly that $g$ solves the homogeneous equation,
\begin{equation}
\overline{a}(v) \cdot \nabla_x g = 0.
\end{equation}
Thus, $h = f - g$ satisfies \eqref{eqn:kinetic2} with the same $\mu$ as $u$. Next, we devise appropriate test functions to show \eqref{eqn:bound1}. 
We fix two cutoff functions $\varphi \in C^\infty_c(B_2)$ and $\zeta \in C^\infty_c(\R)$ such that:
\begin{equation}
\begin{cases}
\varphi(x) = 1 &\text{if }x \in B_r\\
\varphi(x) = 0 &\text{if }x\in B_2 \setminus B_R\\
0 \le \varphi(x) \le 1 &\text{for each }x\in B_2\\
|\nabla\varphi(x)| \le C(R-r)^{-1} &\text{for each }x \in B_2.
\end{cases}\quad \text{and}\quad
\begin{cases}
\zeta_{\eps}(v) = 1 &\text{if } 2\eps < v <\eps^{-1}\\
\zeta_{\eps}(v) = 0 &\text{if }-\infty < v < \eps \text { or } 2\eps^{-1} < v \\
0 \le \zeta_{\eps}(v) \le 1 &\text{for each }v\in \R\\
|\zeta_{\eps}^\prime(v)| \le Cv^{-1} &\text{for each }v\in\R^+.
\end{cases}
\end{equation}
We would like to use $[v-w(x)]_+ \varphi(x)$, which is evidently Lipschitz, 
as a test function in \eqref{eqn:kinetic2} to obtain \eqref{eqn:bound1}. However, because of the relatively arbitrary Borel measure $\mu$, we need to ensure that the $v$-derivative is continuous by smoothly cutting out a neighborhood of $\set{v=w(x)}$. More precisely our test function is: 
\begin{equation}
\Phi_{\eps}(x,v) := \varphi(x)\zeta_{\eps}(v - w(x))[v - w(x)]_+.
\end{equation}
Using $\Phi_{\eps}(x,v)$ as a test function in \eqref{eqn:kinetic2} for $h$ and remembering $h$ is supported in $B_2 \times [0,\Lambda]$, we find
\begin{equation}\label{eqn:regularized_equality}
\int_{B_2}\int_0^{\Lambda}\left[\overline{a}(v)\cdot \nabla_x \Phi_{\eps}(x,v)\right] h(x,v) \;dvdx  = \int_{B_2}\int_0^{\infty} \partial_v\Phi_{\eps} \;d\mu(x,v).
\end{equation}
We first note that 
\begin{equation}
\partial_v \Phi_{\eps}(x,v) = \varphi\zeta_{\eps}^\prime(v - w(x))[v - w(x)]_+ + \varphi\zeta_{\eps}(v - w(x))\chi_{(w(x),\infty)}(v).
\end{equation}
Therefore, we see that $|\partial_v \Phi_{\eps}(x,v)| \le C$ uniformly in $\eps > 0$ and $\partial_v \Phi_{\eps}$ converges pointwise everywhere to $\varphi(x)\chi_{(w(x),\infty)}(v)$ as $\eps \rightarrow 0^+$. By the dominated convergence theorem for the measure $\mu$,
\begin{equation}\label{eqn:RHS_bound}
\begin{aligned}
\lim_{\eps \rightarrow 0^+} \int_{B_2}\int_0^{\Lambda} \Phi_{\eps}(x,v) \;d\mu(x,v) &= \int_{B_2}\int_{w(x)}^{\Lambda} \varphi(x)  \;d\mu(x,v) \ge \mu\left( \set{(x,v) \mid x \in B_r, \ v\in (w(x), \Lambda)}\right).
\end{aligned}
\end{equation}
On the other hand, a similar computation shows $|\nabla_x \Phi_{\eps}(x,v)| \le C$ uniformly in $\eps$ on $B_2 \times [0,\Lambda]$ and
\begin{align*}
\nabla_x \Phi_{\eps}(x,v) &= [v-w(x)]_+\nabla_x \varphi(x) \zeta_{\eps}(v-w(x)) - [v-w(x)]_+\varphi(x)\zeta_{\eps}^\prime(v - w(x))\nabla w(x)\\
    &\qquad - \varphi(x)\zeta_{\eps}(v-w(x))\chi_{(w(x),\infty)}(v)\nabla_xw(x)\\
	&\rightarrow [v-w(x)]_+\nabla_x \varphi(x) - \varphi(x)\chi_{(w(x),\infty)}(v)\nabla_xw(x),
\end{align*}
pointwise for Lebesgue almost every $x,v$. By the dominated convergence theorem, we conclude the left hand side of \eqref{eqn:regularized_equality} converges to
\begin{equation}\label{eqn:LHS_bound}
\begin{aligned}
\lim_{\eps \rightarrow 0^+} LHS &= \int_{B_2} \int_{w(x)}^{\Lambda} \overline{a}(v) \cdot \left( [v-w(x)]_+\nabla_x \varphi(x) - \varphi(x)\nabla_x w(x) \right) h(x,v) \;dvdx\\
	&\le C(\Lambda, n)\|\overline{a}\|_{L^\infty(0,\Lambda)}\left(\|w\|_{Lip(B_R)} + \frac{1}{R-r}\right)\int_{B_R}\int_{w(x)}^{\Lambda} |h(x,v)| \;dvdx.
\end{aligned}
\end{equation}
Combining \eqref{eqn:RHS_bound} and \eqref{eqn:LHS_bound} proves \eqref{eqn:bound1}. To prove \eqref{eqn:bound2}, we redefine $\Phi_{\eps}(x,v)$ as
\begin{equation}
\Phi_{\eps}(x,v) := \varphi(x)\zeta_{\eps}(w(x) - v) [w(x) - v]_+,
\end{equation}
and apply the same argument using $h$ is supported in $B_2 \times [0,\Lambda]$.
\end{proof}

\medskip

In the standard De Giorgi iteration method, one estimates the level set functions $(u-k)_+$ for $k \ge 0$, which are naturally positive subsolutions to the original equation. However, as observed earlier, the functions $(u - \tilde u - k)_+$ for $k \ge 0$ are no longer subsolutions. One of the main ideas in the proof of Theorem \ref{thm:stability} is to replace the standard  level set functions with a kinetic analogue, namely,
\begin{equation}
(u(x) - \tilde u(x) - k)_+ = \int_{\tilde u(x) + k}^{\Lambda} h(x,v) \;dv.
\end{equation}
This motivates the following estimate, in which we combine our total variation estimate on $\mu$ with the averaging lemmas from Theorem \ref{thm:averaging} to find an energy estimate on these functions:
\begin{lemma}[Energy Estimate]\label{lem:energy}
Suppose $u$ and $\tilde{u}$ are bounded entropy solutions to \eqref{eqn:scalar2} on $B_2\subset \R^{n+1}$ with $0 \le u, \tilde u \le \Lambda$ on $B_2$ and $\tilde u$ Lipschitz with $\|\tilde u\|_{Lip(B_2)} \le \Gamma$. Let $f$ and $g$ be the kinetic functions associated to $u$ and $\tilde u$, respectively, and $h = f - g$.
Fix $\theta(\alpha) = \alpha/(\alpha + 4)$, where $\alpha$ and $C_0, \ C_1 > 0$ are the parameters from the genuine nonlinearity condition \eqref{eqn:nondegeneracy}. For any $0 < \beta < \frac{1 - \theta}{2} + \frac{\theta}{n+1}$ for any $0 < r < R < 2$ and $0 \le \ell_1 < \ell_2 \le \Lambda$, we have
\begin{align}
\int_{B_r}\int_{\tilde{u} + \ell_2}^{\Lambda} h \;dvdx &\le \frac{C(\Lambda, C_0, C_1, \alpha, \beta, n)}{(\ell_2 - \ell_1)^\theta}\left(\Gamma + \frac{1}{R - r}\right)^\theta\\
    &\times \left(\int_{B_R}\int_{\tilde{u} + \ell_1}^{\Lambda} h \;dvdx\right)^{\frac{1 + \theta}{2}}\abs{B_r \cap \set{\int_{\tilde u(x) + \ell_2}^{\Lambda} h(x,v) \;dv > 0}}^\beta\label{eqn:energy1}\\[7pt]
\int_{B_r}\int_{-\Lambda}^{\tilde{u} - \ell_2} |h| \;dvdx &\le \frac{C(\Lambda, C_0, C_1, \alpha, \beta, n)}{(\ell_2 - \ell_1)^\theta}\left(\Gamma + \frac{1}{R - r}\right)^\theta\\
    &\times \left(\int_{B_R}\int_{-\Lambda}^{\tilde{u} - \ell_1} |h| \;dvdx\right)^{\frac{1 + \theta}{2}}\abs{B_r \cap \set{\int_{-\Lambda}^{\tilde u(x) - \ell_2} |h(x,v)| \;dv > 0}}^\beta \label{eqn:energy2}
\end{align}
\end{lemma}

\begin{proof}
Set $r^* = (R + r)/2$ and define $\psi\in C^\infty_c(B_2\times \R^+)$ a cut-off function satisfying  
\begin{equation}
\begin{cases}
\psi(x,v) = 1 &\text{if } v > \tilde{u}(x) + \ell_2 \text{ and } x\in B_{r^*}\\
\psi(x,v) = 0 &\text{if }v < \tilde{u}(x) + \ell_1 \text{ or } x\in B_2 \setminus B_R\\
0 \le \psi(x,v) \le 1 &\text{for each }x\in B_2, \ v\in \R^+\\
|\partial_v \psi | \le C(\ell_2 - \ell_1)^{-1} &\text{for each }x\in B_2, \ v\in\R^+.\\
|\nabla_x \psi| \le C(R-r^*)^{-1} &\text{for each }x\in B_2, \ v \in\R^+.
\end{cases}
\end{equation}
Then, we localize $h$ by considering the equation for $\tilde h := \psi(x,v)h(x,v)$, namely,
\begin{equation}
\overline{a}(v)\cdot \nabla_x\tilde h = \partial_v \mu_1 + \mu_0 + H,
\end{equation}
where
\begin{equation}
\mu_1 := \psi(x,v)\mu(x,v), \quad \mu_0 := -\partial_v\psi(x,v)\mu(x,v), \quad \text{and} \quad H := h(x,v)\overline{a}(v)\cdot \nabla_x \psi.
\end{equation}
Applying Theorem \ref{thm:averaging} to $\tilde h$, we conclude
\begin{equation}\label{eqn:average_estimate}
\left\|\int_0^{\Lambda} \tilde h \;dv\right\|_{W^{s,p}(\R^{n+1})} \le C(C_0, C_1, \alpha, s, n)\|\tilde h\|_{L^2(\R^{n+2})}^{1-\theta}\left(\|\mu_1\|_{TV} + \|\mu_0\|_{TV} + \|H\|_{L^1(\R^{n+2})}\right)^{\theta},
\end{equation}
where 
$$\theta = \frac{\alpha}{\alpha + 4}, \qquad p = \frac{1 + \theta}{2}, \qquad \text{and} \qquad 0 < s < \theta.$$
Next, we bound each term on the right individually. First, we see $h$ is $\{0,1\}$-valued on the set $\set{v > \tilde u}$ so that
\begin{equation}
\|\tilde h\|_{L^2(\R^{n+2})} \le \left(\int_{B_R}\int_{\tilde{u}(x) + \ell_1}^{\Lambda} h(x,v) \;dvdx\right)^{\frac{1}{2}}.
\end{equation}
Second, we bound $H$ as
\begin{equation}
\|H\|_{L^1} = \int_{B_R}\int_{\tilde u(x) +\ell_1}^{\Lambda} |h(x,v)\overline{a}(v)\cdot \nabla_x \psi(x,v)| \;dvdx \le \frac{C}{R - r^*}\int_{B_R}\int_{\tilde u(x) + \ell_1}^{\Lambda} h(x,v) \;dvdx.
\end{equation}
Third, we bound $\mu_i$ using Lemma \ref{lem:variation} as
\begin{equation}
\begin{aligned}
\|\mu_i\|_{TV} &\le \frac{C}{(\ell_2 - \ell_1)^i}\mu\left( \set{(x,v) \mid x \in B_{r^*}, \ v\in (\tilde u(x) + \ell_1, \Lambda)}\right)\\
	& \le \frac{C}{(\ell_2 - \ell_1)^i}\left(\Gamma + \frac{1}{r^*-r}\right)\int_{B_{r^*}}\int_{\tilde u(x) +\ell_1}^{\Lambda} h(x,v) \;dvdx.
\end{aligned}
\end{equation}
Finally, we lower bound the left hand side of \eqref{eqn:average_estimate} using the Sobolev embedding $W^{s,p}(\R^{n+1}) \embeds L^q(\R^{n+1})$ and H\"older's inequality to obtain
\begin{equation}
\begin{aligned}
\norm{\int_0^{\Lambda}\tilde h(x,v) \;dv}_{W^{s,p}(\R^{n+1})} &\ge \norm{\int_0^{\Lambda} \tilde h(x,v) \;dv}_{L^q(\R^{n+1})}\\
	& \ge \norm{\int_{\tilde{u}(x) + \ell_2}^{\Lambda} h(x,v) \;dv}_{L^q(B_r)}\\
	&\ge \left(\int_{\tilde{u}(x) + \ell_2}^{\Lambda} h(x,v) \;dvdx\right)\abs{\set{ x\in B_r \ \bigg | \ \int_{\tilde{u}(x) + \ell_2}^{\Lambda} h(x,v)\;dv > 0}}^{\frac{-1}{q^\prime}},
\end{aligned}
\end{equation}
Rearranging terms, yields
\begin{equation}
\int_{B_r}\int_{\tilde{u} + \ell_2}^{\Lambda} h \;dvdx \le \frac{C}{(\ell_2 - \ell_1)^\theta}\left(\Gamma + \frac{1}{R-r}\right)^\theta\left(\int_{B_R}\int_{\tilde{u} + \ell_1}^{\Lambda} h \;dvdx\right)^{\frac{1 + \theta}{2}}\abs{B_r \cap \set{\int_{\tilde{u}(x) + \ell_1}^{\Lambda} h(x,v) \;dv > 0}}^{\frac{1}{q^\prime}}
\end{equation}
where 
\begin{equation}
\theta = \frac{\alpha}{\alpha + 4} \quad\text{and}\quad p = \frac{1 + \theta}{2} \quad\text{and}\quad q = \frac{p(n+1)}{n+1 - ps} \quad\text{and}\quad q^\prime = \frac{q}{q-1} \quad \text{and}\quad 0 < s < \theta.
\end{equation}
Noticing that as $s$ approaches $\theta$, $1/q^\prime = (1-\theta)/2 + s/(n+1)$ tends to $(1 - \theta)/2 + \theta/(n+1)$ completes the proof of \eqref{eqn:energy1}. The proof of \eqref{eqn:energy2} is identical up to using \eqref{eqn:bound2} instead of \eqref{eqn:bound1} to bound the variation of $\mu_i$. 
\end{proof}

We are now ready to prove Theorem \ref{thm:stability} at the unit scale, still in the time-independent setting of \eqref{eqn:scalar2}: 
\begin{lemma}[De Giorgi Iteration]\label{lem:DeGiorgi}
Suppose $u$ and $\tilde{u}$ are bounded entropy solutions to \eqref{eqn:scalar2} on $B_2\subset \R^{n+1}$ with $0 \le u, \tilde u \le \Lambda$ on $B_2$ and $\tilde u$ Lipschitz with $\|\tilde u\|_{Lip(B_2)} \le \Gamma$. 
There are $\gamma = \gamma(\alpha,n) \in (0,1)$, $\eta \in (0,\alpha/2)$,  and $\tilde C = \tilde C(C_0, C_1, \alpha, n, \Lambda) > 0$ such that 
\begin{align}
&\| [u - \tilde u]_+\|_{L^{\infty}(B_1)} \le \tilde C(1 + \Gamma)^{\eta}\| [u - \tilde u]_+\|_{L^1(B_2)}^\gamma\\
&\| [\tilde u - u]_+\|_{L^{\infty}(B_1)} \le \tilde C(1 + \Gamma)^{\eta}\| [\tilde u - u]_+\|_{L^1(B_2)}^\gamma
\end{align}
\end{lemma}

\begin{proof}
Let us set up the iteration quantities. For $0 < K \le \Lambda$, we define
\begin{equation}
r_k = 1 + 2^{-k} \quad \text{and} \quad B_k = B_{r_k} \quad \text{and} \quad \ell_k = K(1 - 2^{-k}) \quad \text{and} \quad A_k = \int_{B_k}\int_{\tilde u(x) + \ell_k} h(x,v) \;dvdx.
\end{equation}
Now, directly applying Lemma \ref{lem:energy} with $r = r_{k+1}$ and $R = r_k$ and some $0 < \beta = (1 - \theta)/2 + \delta$ where $0 < \delta < \theta/(n+1)$, we obtain the recurrence
\begin{equation}
A_{k+1} \le \frac{C(\Lambda,C_0, C_1, \alpha, n)}{(\ell_{k+1} - \ell_k)^\theta}\left(\Gamma + \frac{1}{r_k - r_{k+1}}\right)^\theta A_k^{\frac{1 + \theta}{2}}\abs{B_{k+1} \cap \set{\int_{\tilde u(x) + \ell_{k+1}}^{\Lambda} h(x,v) \;dv > 0}}^\beta.
\end{equation}
It remains to bound the measure of the set appearing on the right hand side. To this end, suppose $x\in B_{k+1}$ with $\int_{\tilde u(x) + \ell_{k+1}}^{\Lambda} h(x,v)\;dv > 0$. Then, $u(x) > \tilde u(x) + \ell_{k+1}$. Thus, $h(x,v) = 1$ for $v\in (\tilde u(x),\ell_{k+1})$. We conclude
$$\int_{\tilde u(x) + \ell_{k}}^{\Lambda} h(x,v) \;dv > \int_{\tilde u(x) + \ell_{k}}^{\tilde u(x) + \ell_{k+1}} 1 \;dv = \ell_{k+1} - \ell_k = K 2^{-k-1}.$$
Thus, applying Chebychev's inequality,
\begin{equation}
\abs{B_{k+1} \cap \set{\int_{\tilde u(x) + \ell_{k+1}}^{\Lambda} h(x,v) \;dv > 0}} = \abs{B_{k+1} \cap \set{\int_{\tilde u(x) + \ell_{k}}^{\Lambda} h(x,v) \;dv > K2^{-k-1}}} \le K^{-1}2^{k+1}A_k
\end{equation}
and we find
\begin{equation}\label{eqn:recurrence}
A_{k+1} \le \frac{C2^{(2\theta + \beta)(k+1)}}{K^{\beta + \theta}}\left(1 + \Gamma\right)^\theta A_k^{1 + \delta}.
\end{equation}
Rearranging terms, setting $\gamma^{-1} := (\beta + \theta)/\delta \in (0,1)$, and $\eta := \theta/(\beta + \theta) \in (0, \alpha/2)$, we find
\begin{equation}
\frac{(1+\Gamma)^{\theta/\delta}A_{k+1}}{K^{1/\gamma}} \le C2^{(2\theta + \beta)(k+1)}\left(\frac{(1 + \Gamma)^{\theta/\delta}A_k}{K^{1/\gamma}}\right)^{1 + \delta}.
\end{equation}
Because $1 + \delta > 1$, $A_k$ decays super-exponentially and $A_\infty = \lim_{k\rightarrow \infty} A_k = 0$, provided $(1+\Gamma)^{\theta/\delta}A_0/K^{1/\gamma}$ is sufficiently small. More precisely, there is a $\tilde C = \tilde C(\Lambda, C_0, C_1, \alpha, n)$ such that if $(1+\Gamma)^{\theta/\delta}A_0/K^{1/\gamma} \le \tilde C^{1/\gamma}$, then $A_\infty = 0$.
Picking $K =\tilde C(1+\Gamma)^{\eta}A_0^{\gamma}$, yields
$$0 = A_\infty = \int_{B_1}\int_{\tilde u(x) + K}^{\Lambda} h(x,v)\;dv$$
and so
$$u(x) \le \tilde u(x) + K \le \tilde u(x) + \tilde C \left(1 + \Gamma\right)^\eta \left(\int_{B_2}\int_{\tilde u(x)}^{\Lambda} h(x,v)\;dv\right)^\gamma.$$
Rearranging terms gives the desired estimate on $[u - \tilde u]_+$. The proof of the lower bound is identical using the other energy estimate from Lemma \ref{lem:energy} and similar reasoning to arrive at a recurrence like \eqref{eqn:recurrence}.
\end{proof}

\begin{remark}
For time dependent equations like \eqref{eqn:scalar}, there is a order of causality and one can show the more precise estimate
\begin{equation}
    \norm{u - \tilde u}_{L^\infty([-1,0]\times B_1)} \le C\left(1 + \norm{\tilde u}_{Lip([-2,0]\times B_2)}\right)^\eta\norm{u - \tilde u}_{L^1([-2,0]\times B_2)}^\gamma.
\end{equation}
This follows from a standard modification of the De Giorgi techniques applied in this section, where one modifies the cut-offs in space and time.
\end{remark}

\begin{flushleft}

\bf{\underline{Proof of Theorem \ref{thm:stability}}}

\end{flushleft}

We will now show how Lemma \ref{lem:DeGiorgi} implies the full Theorem \ref{thm:stability}. Take $0 < s < t < T$,
suppose $\Omega$ is an a bounded open subset of $\R^n$ with $\tilde \Omega$ compactly contained in $\Omega$, and let $u$ and $\tilde u$ be bounded entropy solutions to \eqref{eqn:scalar} on $[0,T]\times \Omega$ as in the statement of Theorem \ref{thm:stability}. We will now analyze $u - \tilde u$ pointwise using Lemma \ref{lem:DeGiorgi}.

Let us first assume both $u, \tilde u$ are non-negative with $L^\infty$ bound $\Lambda$ on $[s,T]\times \Omega$. Fix $x \in \widetilde \Omega$ and $t < t_* < T$. Let $d = \min(t - s,dist(\widetilde\Omega,\bd\Omega))$ Then, define $w, \tilde w:[-2,0]\times B_2 \rightarrow \R$ via
\begin{equation}
w(\tau, y) = u(t_* + \tau d/2, x + yd/2) \qquad \text{and} \qquad \tilde w(y) = \tilde u(t_* + \tau d/2, x + yd/2).
\end{equation}
By scaling, $w, \tilde w$ are solutions to \eqref{eqn:scalar} in the $(\tau, y)$ variables. Also,
\begin{equation}
    \|w\|_{L^\infty([-2,0]\times B_2)} \le \Lambda \quad \text{and} \quad \|\tilde w\|_{L^\infty([-2,0]\times B_2)} \le \Lambda \quad \text{and} \quad \|\tilde w\|_{Lip([-2,0]\times B_2)} \le \frac{d\Gamma}{2}.
\end{equation}
Applying Lemma \ref{lem:DeGiorgi} to $w$ and $\tilde w$, we find
\begin{equation}
\begin{aligned}
|u(t_*,x) - \tilde u(t_*,x)| \le \|w - \tilde w\|_{L^\infty([-1,0]\times B_1)} &\le C(\Lambda, C_0, C_1, n, \alpha)(1 + d\Gamma)^{\eta}\|w - \tilde w\|_{L^1([-2,0]\times B_2)}^\gamma\\
    &\le \frac{C(1 + d\Gamma)^{\eta}}{d^{\gamma(n+1)}}\|u - \tilde u\|_{L^1([s,T]\times \Omega)}^\gamma
\end{aligned}
\end{equation}
Thus, taking a supremum over $t_* \in [s,T]$ and $x\in \tilde \Omega$ proves Theorem \ref{thm:stability} when $u, \; \tilde u \ge 0$.

Finally, the general case reduces to when both $u,\;\tilde u$ are non-negative. Indeed, defining $$m = \min\left(\essinf_{t\in[s,T],\;x\in \Omega} u(t,x), \essinf_{t\in[s,T]\;x\in \Omega}\tilde u(t,x)\right),$$ both $u - m$ and $\tilde u - m$ are non-negative and we can invoke the preceding part on $u - m$ and $\tilde u - m$.




\section{Decay Rates}\label{sec:applications}

In this section, we show how to use Theorem \ref{thm:stability} to obtain explicit rates of convergence for perturbations of global-in-time Lipschitz solutions. In particular, this result applies to planar rarefaction waves, but is more general.

\begin{flushleft}
\bf{\underline{Proof of Corollary \ref{cor:rate}}}
\end{flushleft}

Fix two global-in-time entropy solutions $u$ and $\tilde u$ to \eqref{eqn:scalar} as in the statement of Corollary \ref{cor:rate} with initial data $u_0$ and $\tilde u_0$, respectively. Namely, take $\tilde u$ a global-in-time Lipschitz entropy solution and $\|u_0\|_{L^\infty},\; \|\tilde u_0\|_{L^\infty} \le \Lambda$ and $u_0 - \tilde u_0 \in L^1(\R^n)$.
Then, the standard maximum principle and $L^1$-contraction estimate for scalar conservation laws (for a precise reference see \cite[Theorem 6.2.3]{dafermos_book}) imply that for each $t \in \R^+$,
\begin{equation}
\|u(t)\|_{L^\infty(\R^n)} \le \Lambda, \qquad \|\tilde u(t)\|_{L^\infty(\R^n)} \le \Lambda, \qquad \text{and} \qquad \norm{u(t) - \tilde u(t)}_{L^1(\R^n)} \le \norm{u_0 - \tilde u_0}_{L^1(\R^n)}.
\end{equation}
Now, fix $t \in \R^+$ and $x\in \R^n$ and apply Theorem \ref{thm:stability} with $T = 2t$,  $s = t/2$, $\Omega = B_{t}(x)$ and $\widetilde \Omega = B_{t/2}(x)$:
\begin{equation}
    |u(t,x) - \tilde u(t,x)| \le \frac{C(\Lambda, C_0, C_1, \alpha, n)(1 + d\Gamma)^{\eta}}{d^{\gamma(n+1)}}\norm{u - \tilde u}_{L^1([s,T]\times B_t(x))}^{\gamma},
\end{equation}
where $d = \min(t-s,dist(\bd\Omega,\widetilde\Omega)) = t/2$ and $\Gamma = \|\nabla_{t,x}\tilde u\|_{L^\infty([t/2,T]\times B_t(x))}$. Thus, we have
\begin{equation}
    \abs{u(t,x) - \tilde u(t,x)} \le \frac{C\left(1 + t\sup_{t/2 < \tau < 2t}\norm{\nabla_{t,x}\tilde u(\tau)}_{L^\infty(\R^n)}\right)^{\eta}}{t^{\gamma n}}\norm{u_0 - \tilde u_0}_{L^1(\R^n)}^{\gamma}.
\end{equation}
Finally, using $\tilde u$ is Lipschitz, $\tilde u$ satisfies \eqref{eqn:scalar} pointwise. One can readily show $\abs{\partial_t \tilde u} \le C(C_1, \Lambda)\abs{\nabla \tilde u}$, which completes the proof of Corollary \ref{cor:rate}.



\bibliographystyle{plain}
\bibliography{bibliography}

\begin{thebibliography}{10}

\bibitem{chen_frid2}
G.-Q. Chen and H.~Frid.
\newblock Decay of entropy solutions of nonlinear conservation laws.
\newblock {\em Arch. Ration. Mech. Anal.}, 146(2):95--127, 1999.

\bibitem{crippa_otto_westdickenberg}
G.~Crippa, F.~Otto, and M.~Westdickenberg.
\newblock Regularizing effect of nonlinearity in multidimensional scalar
  conservation laws.
\newblock In F.~Ancona, S.~Bianchini, R.~M. Colombo, C.~De~Lellis, A.~Marson,
  and A.~Montanari, editors, {\em Transport equations and multi-{D} hyperbolic
  conservation laws}, volume~5 of {\em Lect. Notes Unione Mat. Ital.}, pages
  77--128. Springer, Berlin, Heidelberg, 2008.

\bibitem{dafermos}
C.~Dafermos.
\newblock The second law of thermodynamics and stability.
\newblock {\em Arch. Ration. Mech. Anal.}, 70(2):167--179, 1979.

\bibitem{dafermos3}
C.~Dafermos.
\newblock Regularity and large time behaviour of solutions of a conservation
  law without convexity.
\newblock {\em Proc. Roy. Soc. Edinburgh Sect. A}, 99(3-4):201--239, 1985.

\bibitem{dafermos2}
C.~Dafermos.
\newblock Long time behavior of periodic solutions to scalar conservation laws
  in several space dimensions.
\newblock {\em SIAM J. Math. Anal.}, 45(4):2064--2070, 2013.

\bibitem{dafermos_book}
C.~Dafermos.
\newblock {\em Hyperbolic conservation laws in continuum physics}, volume 325
  of {\em Grundlehren der Mathematischen Wissenschaften [Fundamental Principles
  of Mathematical Sciences]}.
\newblock Springer-Verlag, fourth edition, 2016.

\bibitem{degiorgi}
E.~De~Giorgi.
\newblock Sulla differenziabilit{\`a} e l'analiticit{\`a} delle estremali degli
  integrali multipli regolari.
\newblock {\em Mem. Accad. Sci. Torino. Cl. Sci. Fis. Mat. Nat.}, 3(3):25--43,
  1957.

\bibitem{delellis_otto_westdickenberg}
C.~De~Lellis, F.~Otto, and M.~Westdickenberg.
\newblock Structure of entropy solutions for multi-dimensional scalar
  conservation laws.
\newblock {\em Arch. Ration. Mech. Anal.}, 170(2):137--184, 2003.

\bibitem{debussche_vovelle}
A.~Debussche and J.~Vovelle.
\newblock Long time behavior in scalar conservation laws.
\newblock {\em Differential Integral Equations}, 22(3-4), 2009.

\bibitem{diperna}
R.~DiPerna.
\newblock Uniqueness of solutions to hyperbolic conservation laws.
\newblock {\em Indiana Univ. Math. J.}, 28(1):137--188, 1979.

\bibitem{enquist_e}
B.~Enquist and W.~E.
\newblock Large time behavior and homogenization of solutions of
  two-dimensional conservation laws.
\newblock {\em Comm. Pure Appl. Math.}, 46(1):1--26, 1993.

\bibitem{golding}
W.~Golding.
\newblock Unconditional regularity and trace results for the isentropic {E}uler
  equations with {$\gamma = 3$}.
\newblock {\em ar{X}iv preprint: 2207.05821}, 2022.

\bibitem{hoff}
D.~Hoff.
\newblock The sharp form of {O}leinik’ s entropy condition in several space
  variables.
\newblock {\em Trans. Amer. Math. Soc.}, 276(2):707--714, 1983.

\bibitem{jenssen_sinestrari}
H.~K. Jenssen and C.~Sinestrari.
\newblock On the spreading of characteristics for non-convex conservation laws.
\newblock {\em Proc. Roy. Soc. Edinburgh Sect. A}, 131(4):909--925, 2001.

\bibitem{kruzkov}
S.~N. Kru{\v{z}}kov.
\newblock First order quasilinear equations in several independent variables.
\newblock {\em Mat. Sb.}, 10(2):217--243, 1970.

\bibitem{lax}
P.~Lax.
\newblock Hyperbolic systems of conservation laws. {II}.
\newblock {\em Comm. Pure Appl. Math.}, 10:537--566, 1957.

\bibitem{lions_perthame_tadmor1}
P.-L. Lions, B.~Perthame, and E.~Tadmor.
\newblock A kinetic formulation of multidimensional scalar conservation laws
  and related equations.
\newblock {\em J. Amer. Math. Soc.}, 7(1):169--191, 1994.

\bibitem{oleinik}
O.~A. Oleinik.
\newblock Discontinuous solutions of nonlinear differential equations.
\newblock {\em Transl. AMS Ser. 2}, 26:95--172, 1963.

\bibitem{panov2}
E.~Y. Panov.
\newblock On decay of periodic entropy solutions to a scalar conservation law.
\newblock {\em Ann. Inst. H. Poincar{\'e} C Anal. Non Lin{\'e}aire},
  30(6):997--1007, 2013.

\bibitem{serre3}
D.~Serre.
\newblock Long-time stability in systems of conservation laws, using relative
  entropy/energy.
\newblock {\em Arch. Ration. Mech. Anal.}, 219:679--699, 2016.

\bibitem{serre1}
D.~Serre.
\newblock Divergence-free positive symmetric tensors and fluid dynamics.
\newblock {\em Ann. Inst. H. Poincare{\'{e}} C Anal. Non Lin{\'e}aire},
  35(5):1209--1234, 2018.

\bibitem{serre2}
D.~Serre.
\newblock Compensated integrability. {A}pplications to the {V}lasov-{P}oisson
  equation and other models in mathematical physics.
\newblock {\em J. Math. Pures Appl. (9)}, 127:67--88, 2019.

\bibitem{serre_silvestre}
D.~Serre and L.~Silvestre.
\newblock Multi-dimensional {B}urgers equation with unbounded initial data:
  well-posedness and dispersive estimates.
\newblock {\em Arch. Ration. Mech. Anal.}, 234(3):1391--1411, 2019.

\bibitem{silvestre}
L.~Silvestre.
\newblock Oscillation properties of scalar conservation laws.
\newblock {\em Comm. Pure Appl. Math.}, 72(6):1321--1348, 2019.

\bibitem{tadmor_tao}
E.~Tadmor and T.~Tao.
\newblock Velocity averaging, kinetic formulations, and regularizing effects in
  quasi-linear {PDE}s.
\newblock {\em Comm. Pure Appl. Math.}, 60:1488--1521, 2007.

\bibitem{tartar}
L.~Tartar.
\newblock Compensated compactness and applications to partial differential
  equations.
\newblock In {\em Nonlinear analysis and mechanics: Heriot-Watt Symposium, Vol.
  IV}, volume~39 of {\em Res. Notes in Math.}, pages 136--212. Pitman, Boston,
  MA, 1979.

\end{thebibliography}

\end{document}